\documentclass[11pt]{amsart}
\usepackage{amsfonts,amssymb,amsbsy,amsmath,amscd,amstext,amsthm}

\usepackage{graphicx}

\newtheorem{theorem}{Theorem}
\newtheorem{thm}[theorem]{Theorem}

\newtheorem{proposition}{Theorem}
\newtheorem{lem}[proposition]{Lemma}

\newtheorem{cor}[theorem]{Corollary}

\theoremstyle{remark}

\def\Z{\mathbb Z}

\begin{document}

\title{Mean Divisibility of Multinomial coefficients}
\author{Shigeki Akiyama}
\address{Institute of Mathematics\\ University of Tsukuba \\
1-1-1 Tennodai, Tsukuba, Ibaraki, 305-8571 JAPAN}
\email{akiyama@math.tsukuba.ac.jp}
\date{}
\thanks{The author is supported by the Japanese Society for the Promotion of Science (JSPS), Grant in aid 24540012.}

\begin{abstract}
Let $m_1,\dots,m_s$ 
be positive integers. 
Consider the sequence 
defined by multinomial coefficients:
$$
a_n=\binom{(m_1+m_2+\dots +m_s)n}{m_1 n, m_2 n, \dots, m_s n}.
$$
Fix a positive integer $k\ge 2$. 
We show that there exists a positive integer $C(k)$ such that
$$
\frac{\prod_{n=1}^t a_{kn}}{\prod_{n=1}^t a_n} \in \frac 1{C(k)} \Z
$$
for all positive integer $t$, if and only if $\mathop{GCD}(m_1,\dots,m_s)=1$.
\end{abstract}

\maketitle

\section{Mean Divisibility}

A sequence $(a_n)\ (n=1,2,\dots)$ of non zero integers 
is {\it divisible} if $n \mid m$ implies $a_n \mid a_m$. It 
is {\it strongly divisible} if $\mathop{GCD}(a_n,a_m)
=|a_{\mathop{GCD}(n,m)}|$. 
Such divisibility attracts number theorists for a long time and a lot of
papers dealt with properties of such sequences
\cite{Stewart:77,Schinzel:87,BiluHanrotVoutier,BezivinPethoPoorten, Flatters:09,Akiyama:96,InghamMaheSilvermanStangeStreng}.
Primitive divisors of elliptic divisibility sequences and sequences arose
in arithmetic dynamics
are recently studied in detail \cite{EinsiedlerEverestWard, InghamSilverman:09,
VoutierYabuta:12}.
In this paper, we introduce a weaker terminology which seems not studied before. 
We say that $(a_n)$ is {\it almost mean $k$-divisible}, if there is a positive integer
$C=C(k)$ such that $(\prod_{n=1}^t a_{kn})/(\prod_{i=1}^t a_n)\in \frac 1C \Z$ 
for any positive integer $t$. In particular,  
$(a_n)$ is {\it mean divisible}
if $\prod_{n=1}^t a_{n} \mid \prod_{i=1}^t a_{kn}$ for any positive integer $k$ and
$t$. 
Clearly if $(a_n)$ is divisible, 
then it is mean divisible. By definition, if 
a sequence is almost mean $k$-divisible
for all $k$ with the constant $C(k)=1$, then it is mean divisible.
We are interested in giving non trivial examples 
of (almost) mean divisible sequences. 
In fact, we
show that sequences defined by multinomial coefficients give such examples.
Let $m_1,\dots,m_s$ be positive integers. A {\it multinomial sequence}
is defined by
$$
a_n=\binom{(m_1+m_2+\dots +m_s)n}{m_1 n, m_2 n, \dots, m_s n}
=\frac {((m_1+m_2+\dots +m_s)n)!}{(m_1n)!(m_2n)!\dots (m_sn)!}.
$$

\begin{thm}
\label{Main}
If $\mathop{GCD}(m_1,m_2,\dots, m_s)=1$, then the
multinomial sequence is almost mean $k$-divisible for all $k$.
\end{thm}

The proof relies on an interesting integral inequality (Lemma \ref{MainIneq}) 
and its approximation by Riemann sums. Here are some illustrations:

\begin{cor}
\label{Co1}
$$
\frac{\prod_{n=1}^t \binom{10n}{4n}}{\prod_{n=1}^t \binom{5n}{2n}}\in \frac 1{11}\Z,
\qquad
\frac{\prod_{n=1}^t \binom{9n}{3n}}{\prod_{n=1}^t \binom{3n}{n}}\in \frac 1{5}\Z,
\qquad
\frac{\prod_{n=1}^t \binom{28n}{4n,8n,16n}}{\prod_{n=1}^t \binom{7n}{n,2n,4n}}\in \Z
$$
for any positive integer $t$.
\end{cor}
Readers will see that Figures \ref{23},\ref{12} and \ref{124} in \S \ref{Comp}
essentially tell why these are true. 
The constant $C(k)$ is computed 
by an algorithm based on the proof
of Theorem \ref{Main}. 
However it is not so easy to identify the set of $t$'s at which 
the denominator actually appears.
For the first example, there are infinitely many $t$ with denominator $11$, 
but the denominator $5$ in the second example appears only when $t=2$. 
See \S \ref{Comp} for details. We can also show

\begin{thm}
\label{Main2}
If $\mathop{GCD}(m_1,m_2,\dots, m_s)>1$, then the
multinomial sequence is not almost mean $k$-divisible for all $k$.
\end{thm}

Thus for a given $k$, a multinomial sequence is almost mean $k$-divisible
if and only if $\mathop{GCD}(m_1,m_2,\dots, m_s)=1$ holds. For e.g.,

\begin{cor}
\label{Co2}
Denominators of 
$$
\frac{\prod_{n=1}^t \binom{8n}{4n}}{\prod_{n=1}^t \binom{4n}{2n}}
$$
for $t=1,2,\dots$ forms an infinite set. 
\end{cor}

The proof of Theorem \ref{Main2}
uses the Riemann sum approximation again, but 
we have to study more precisely the integral inequality of Lemma \ref{MainIneq} 
at the place where it attains the equality. Indeed, we show that the set of
 primes in some arithmetic progression modulo $k \mathop{LCM}(m_1,\dots,m_s,\sum_{i=1}^s m_i)$
must appear in the denominators.
If $(a_n)$ is a multinomial sequence
with parameters $(m_1,\dots, m_s)$,
then so is $(a_{\ell n})$ with $(\ell m_1,\dots, \ell m_s)$. Since divisibility of
$(a_n)$ is hereditary to $(a_{\ell n})$, we see from Theorem \ref{Main2}, 

\begin{cor}
\label{NonDivisible}
Multinomial sequences are not divisible.
\end{cor}

Therefore multinomial sequences supply non trivial examples 
of almost mean $k$-divisible sequences. 
For the mean divisibility, we can prove

\begin{thm}
\label{Mean}
If $m_1,\dots,m_s$ are pairwise coprime and  
each $m_i$ divides $\sum_{i=1}^s m_i$, then the
multinomial sequence is mean divisible.
\end{thm}

For examples, we have

\begin{cor}
\label{Co3}
$$
\frac{\prod_{n=1}^t \binom{2kn}{kn}}{\prod_{n=1}^t \binom{2n}{n}}\in  \Z,
\qquad
\frac{\prod_{n=1}^t \binom{3kn}{kn,kn,kn}}{\prod_{n=1}^t \binom{3n}{n,n,n}}\in \Z,
\qquad
\frac{\prod_{n=1}^t \binom{6kn}{kn,2kn,3kn}}{\prod_{n=1}^t \binom{6n}{n,2n,3n}}\in \Z
$$
for any positive integer $k$ and $t$.
\end{cor}

The central binomial coefficient $\binom{2n}{n}$ is 
of historical importance. By using $\binom{2n}{n}$, Erd\H{o}s \cite{Erdoes:32}
showed 
Bertran-Chebyshev's theorem that there is a prime in any interval $(n,2n]$.
Interesting divisibility problems on $\binom{2n}{n}$
are discussed in \cite{Moser:63, Erdoes:64, Moshe:06}.
However the first example of Corollary \ref{Co3} seems to be new.

Theorem \ref{Mean} models Theorem 5.2 in \cite{Akiyama-Petho:12_1} which proves
$$
2^{-m} \frac{\prod_{j=m+1}^{2m} {2j \choose j}}{\prod_{j=1}^{m-1} {2j \choose j}}
\in \Z, \quad
2^{-m-1} \frac{\prod_{j=m+1}^{2m+1} {2j \choose j}}{\prod_{j=1}^{m} {2j \choose j}}
\in \Z.
$$
Indeed the first example of Corollary \ref{Co3} follows from Lemma 5.2 in \cite{Akiyama-Petho:12_1} as well.

Several further questions are exhibited in \S \ref{Open}. 
The referee of this paper pointed out that, questions around
divisibility of 
multinomial coefficients have long history and are widely studied still now. 
In relation to the present article, we just quote \cite{Landau:00, Granville:97, Bober:09, Sun:12}. Readers find many related works therein. 

\section{Some Lemma}
Let $m_1,\dots,m_s$ and $k$ 
be positive integers.
Put $m=\sum_i^s {m_i}$ and $$
f(x)=\lfloor mx\rfloor - \sum_{i=1}^s
\lfloor m_i x \rfloor,$$
which is the function used in \cite{Landau:00}.
We clearly have \begin{equation}
\label{Period}
f(x+1)=f(x). \end{equation}
From $x-1<\lfloor x\rfloor\le x$, we see
\begin{equation}\label{01}
\lfloor x\rfloor+\lfloor -x \rfloor = 
\begin{cases} -1 & x\not \in \Z\\
              0  & x\in \Z
              \end{cases}.
\end{equation}
Using (\ref{Period}) and (\ref{01}), we obtain 
\begin{equation}
\label{Sym0}
f(x)+f(1-x)=s-1
\end{equation}
unless $mx$ or $m_ix$ is an integer.
Thus we have
$$
\int_{0}^u f(x)dx + \int_{1-u}^1 f(x)dx = (s-1)u.
$$
From this equality, 
$$
\int_{0}^u \left(f(kx)-f(x)\right)dx + \int_{1-u}^1 \left(f(kx)-f(x)\right) dx =0.
$$
Therefore we derive
\begin{equation}
\label{Zero}
\int_{0}^1 \left(f(kx)-f(x)\right)dx=0
\end{equation}
and 
\begin{equation}
\label{Sym}
\int_{0}^u \left(f(kx)-f(x)\right)dx =\int_{0}^{1-u} \left(f(kx)-f(x)\right) dx.
\end{equation}

First we assume $s=2$.
\begin{lem}
\label{2Ineq}
Let $m_1,m_2$ be coprime integers and define $$f(x)=\lfloor (m_1+m_2)x\rfloor
-\lfloor m_1 x\rfloor - \lfloor m_2 x \rfloor.$$ Then for any positive integer 
$k$ and any positive real $u$, we have
$$
\int_{0}^u f(kx) dx \ge \int_0^u f(x) dx
$$
and the equality holds if and only if $u\in \bigcup_{a=0}^{\infty}
\left[
a-\frac{1}{k(m_1+m_2)}, a+\frac {1}{k(m_1+m_2)} \right]$. 
\end{lem}

Figure \ref{23} and \ref{12} in \S \ref{Open} are
the graphs of 
$
\int_{0}^u \left(f(kx) - f(x) \right) dx
$
in special cases, which may help the reader. 

\begin{proof}
By (\ref{Zero}), 
$$
\int_{0}^u \left(f(kx)-f(x)\right) dx
$$
is invariant
under $u\mapsto u+1$. We show the inequality for $u\in [0,1]$.
The function $f(x)$ is a right continuous step function with
discontinuities at $\frac 1m\Z$ and $\bigcup_{i=1}^2 \frac 1{m_i}\Z$
with $m=m_1+m_2$.
The discontinuities at $\frac 1m\Z$ gives $+1$ jump 
and $\bigcup_{i=1}^2 \frac 1{m_i}\Z$ gives $-1$ jump. Thus it is clear that
$f(x)=0$ for $x\in [0,1/m)$ and so $f(x)=1$ for $[(m-1)/m,1)$ by
(\ref{Sym0}) and right continuity.
Since both integrands are identically zero, we have
$$
\int_{0}^u f(kx) dx = \int_0^u f(x) dx
$$
for $u\in [0,\frac 1{km}]$ and the same is true 
for $u\in [\frac {km-1}{km},1]$ by (\ref{Sym}). 
By (\ref{Period}), we first describe the distribution of the discontinuities
in $(0,1)$, i.e., 
$$
\left\{ \frac 1{m},\frac 2{m},\dots \frac {m-1}m \right\}
\bigcup 
\left\{ \frac 1{m_1},\frac 2{m_1},\dots,\frac{m_1-1}{m_1},\frac{1}{m_2},\frac{2}{m_2},\dots, \frac{m_2-1}{m_2} \right\}.
$$
Recalling the idea of Farey fractions, let 
$a/b$ and $c/d$ be two non negative rational numbers 
with $a,b,c,d\in \Z$, $(a,b)=(c,d)=1$ and $a/b<c/d$. 
Then 
\begin{equation}
\label{Farey}
\frac ab< \frac {a+c}{b+d} < \frac cd.
\end{equation}
Arrange elements of
$$
B=\left\{ \frac 1{m_1},\frac 2{m_1},\dots,\frac{m_1-1}{m_1},\frac{1}{m_2},\frac{2}{m_2},\dots, \frac{m_2-1}{m_2} \right\}
$$
in the increasing order. If $j_1/m_1$ and $j_2/m_2$ are adjacent, then we find
a fraction $(j_1+j_2)/m$ in between. Considering the cardinality of $B$, 
we notice that there exists exactly one element of $B$ in the interval
$(i/m, (i+1)/m)$ for $i=1,2,\dots, m-2$. For each $i/m$
for $i=1,2,\dots, m-2$, there is a right adjacent discontinuity
of the form either $j_1/m_1$ or $j_2/m_2$. 
For the convenience, we 
formally extend this idea over the discontinuities at $\Z$.
There are no element of
$\bigcup_{i=1}^2 \frac 1{m_i}\Z$ in $((m-1)/m,1)$ and 
we associate $j_1/m_1$ with $j_1=m_1$
for $(m-1)/m$ and $j_2/m_2$ with $j_2=m_2$ for $m/m=1$. In other words, 
we are formally treating in a way that
$$
\left[\frac{m-1}m,1\right)=\left[\frac{m-1}m, \frac{m_1}{m_1}\right) 
\cup \left[\frac mm, \frac{m_2}{m_2}\right)
$$
though the last interval $[m/m, m_2/m_2)$ is empty. 
We extend this convention to all
positive integers by periodicity (\ref{Period}).
Then for each discontinuity of the form $\frac 1m\Z$, there exists exactly one 
right adjacent discontinuity\footnote{Here we think that
the right adjacent discontinuity of $a=\frac{am}{m} \in \frac 1m \Z$ is
the discontinuity $\frac{a m_2}{m_2}\in \frac 1{m_2}\Z$.} in $\bigcup_{i=1}^2 \frac 1{m_i}\Z$.
Thus we see that $f(x)$ is a step function which
takes exactly two values $\{0,1\}$, and for any positive $u$, the integral
$
\int_0^{u} f(x) dx 
$
is computed as a sum of the length of intervals where $f(x)=1$.
These half open intervals have left
end points in $\frac 1m\Z$ and right end points 
in $\{u\} \cup \left(\bigcup_{i=1}^2 \frac 1{m_i} \Z\right)$.

The required inequality is equivalent to
$$
k \int_{0}^u f(x) dx \le \int_{0}^{ku} f(x)dx
$$
and it suffices to show that
$$
G(u)= \int_{0}^{ku} f(x)dx- k \int_{0}^u f(x) dx 
$$
is positive for $\frac 1{mk}<u<\frac{mk-1}{mk}$ by 
(\ref{Zero}) and (\ref{Sym}). Moreover the positivity is clearly true for
$\frac {1}{mk}<u\le \varepsilon$ and $1-\varepsilon\le u<\frac{mk-1}{mk}$
where $\varepsilon$ is the next
discontinuity of $f(kx)$ adjacent to $\frac 1{mk}$.
As $f(x)\in \{0,1\}$ and
$G'(u)= k(f(ku)-f(u))$, we see if $f(u)=1$ then
$G'(u) \le 0$ and if $f(u)=0$ then $G'(u)\ge 0$. The minimum of $G(u)$ is 
attained either at the end points of $[\varepsilon, 1-\varepsilon]$ 
or the point where $G'(u)$ changes its sign from non positive to non negative,
i.e., where $f$ has negative jump.
Thus it is enough to show that $G(u)>0$ for $u\in B$, because the minimum 
must be equal to
$$
\min \{G(\varepsilon),\min_{u\in B} G(u) \}.
$$
Without loss of generality
we may put $u=j_1/m_1$ with  
$j_1=1,2,\dots, m_1-1$. Summing up
the length of intervals where $f(x)=1$, we have
$$
G\left(\frac{j_1}{m_1}\right)=\left(
\sum_{j=1}^{kj_1} \frac {j}{m_1} + \sum_{j=1}^{kj_2+\ell} \frac{j}{m_2}
-\sum_{j=1}^{k(j_1+j_2)+\ell} \frac{j}{m}\right)
- k \left(
\sum_{j=1}^{j_1} \frac {j}{m_1} + \sum_{j=1}^{j_2} \frac{j}{m_2}
-\sum_{j=1}^{j_1+j_2} \frac{j}{m}\right)
$$
with $0\le j_2<m_2$ and $0\le \ell\le k-1$. 
The right side is
$$
\frac{\left(j_2^2 k (k-1)+2 j_2 k \ell+\ell^2+\ell\right)m_1^2
-2 j_1 (j_2(k-1)+\ell)k m_1 m_2+ j_1^2 k(k-1) m_2^2}
   {2 m_1m_2 (m_1+m_2)}
$$
whose numerator is a quadratic form of $m_1$ and $m_2$ with
the discriminant
$$
-4 j_1^2 k \ell (k-\ell-1)
$$
and the coefficient of $m_2^2$ is positive.
Thus if $0<\ell<k-1$ then $G(\frac{j_1}{m_1})>0$. 
For the remaining cases, we have
$$
2 m_1 m_2
   (m_1+m_2)G\left(\frac{j_1}{m_1}\right)=\begin{cases}
k(k-1) (j_2 m_1-j_1 m_2)^2 & \ell=0\\
k(k-1) ((j_2+1) m_1-j_1 m_2)^2 & \ell=k-1
   \end{cases}.
$$
Since $m_1$ and $m_2$ are coprime and $j_1=1,\dots,m_1-1$, 
the right side can not vanish in both cases.
We have shown the lemma.
\end{proof}

We prepare an elementary inequality:

\begin{lem}
\label{GCDineq}
Let $s$ and $b_1,b_2,\dots,b_{s+1}$ be positive integers with 
$\mathop{GCD}(b_1,b_2,\dots,b_{s+1})=1$ and $b=b_1+b_2+\dots+b_{s+1}$. 
Then we have the following inequality:
$$
\prod_{i=1}^{s+1} \left(\mathop{GCD}(b-b_i,b_i) \mathop{GCD}_{{j\neq i} \atop {1\le j\le 
s+1}} b_j \right)
\le \prod_{i=1}^{s+1} \frac {b-b_i}s
$$
\end{lem}

\begin{proof}
The left side is equal to:
$$
\prod_{i=1}^{s+1} \left(\mathop{GCD}(b-b_{i+1},b_{i+1}) 
\mathop{GCD}_{{j\neq i} \atop {1\le j\le s+1}} b_j \right)
$$
with the convention $b_{s+2}=b_1$. Then we see that
$ \mathop{GCD}(b-b_{i+1},b_{i+1})$ and 
$ \begin{displaystyle}\mathop{GCD}_{{j\neq i} \atop {1\le j\le s+1}}\end{displaystyle} 
b_j $
are coprime divisors of $b_{i+1}$. Therefore we have
$$
\prod_{i=1}^{s+1} \left(\mathop{GCD}(b-b_i,b_i) 
\begin{displaystyle}\mathop{GCD}_{{j\neq i} \atop {1\le j\le s+1}} \end{displaystyle}
b_j \right)
\le \prod_{k=1}^{s+1} b_k
$$
On the other hand
$$
\prod_{i=1}^{s+1} (b-b_i) \ge \prod_{i=1}^{s+1} \left(s (\prod_{j\neq i} b_j)^{1/s}
\right)
=s^{s+1} \prod_{k=1}^{s+1} b_k
$$
which proves the inequality.
\end{proof}

We wish to show a generalization of Lemma \ref{2Ineq}.
\begin{lem}
\label{MainIneq}
Let $m_1,\dots, m_s$ be positive integers and put
$$
m=m_1+m_2+\dots+m_s,\qquad g=\mathop{GCD}(m_1,m_2,\dots,m_s)
$$
and $f(x)=\lfloor mx\rfloor -\sum_{i=1}^s \lfloor m_i x\rfloor$. 
Then for any positive integer 
$k$ and any positive real $u$, we have
$$
\int_{0}^u f(kx) dx \ge \int_0^u f(x) dx
$$
and the equality holds if and only if $u\in \bigcup_{a=0}^{\infty}
\left[
\frac ag-\frac{1}{km}, \frac ag+\frac {1}{km} \right]$.
\end{lem}

\begin{proof}
Lemma \ref{2Ineq} shows the case $s=2$ and $g=1$. The case $s=2$ and $g>1$ is
easily shown by applying Lemma \ref{2Ineq} for $m'=m/g$ and $m_i'=m_i/g$.
We assume that the statement is valid until $s (\ge 2)$ and
prove the case $m=m_1+m_2+\dots +m_{s+1}$ and $\mathop{GCD}(m_1,m_2,\dots, m_{s+1})=1$.
The case $\mathop{GCD}(m_1,m_2,\dots, m_{s+1})>1$ follows similarly to 
the case $s=2$.

By Lemma \ref{GCDineq} with $b_i=m_i$, we may assume
\begin{equation}
\label{SmallGCD}
\mathop{GCD}(m_1,m_2,\dots,m_s)
\mathop{GCD}(\sum_{i=1}^s m_i, m_{s+1})
\le \frac {\sum_{i=1}^s m_i}{s}
\end{equation}
without loss of generality by changing indices. By the induction assumption,
$f_1(x)=\lfloor (\sum_{i=1}^s m_i) x \rfloor -\sum_{i=1}^s \lfloor m_i x \rfloor$
satisfies the inequality:
\begin{equation}
\label{First}
\int_{0}^u f_1(k x) dx \ge \int_0^u f_1(x) dx
\end{equation}
and the equality holds if and only if 
$$
x\in \bigcup_{a=0}^{\infty}
\left[
\frac a{g_1}-\frac{1}{k\sum_{i=1}^s m_i}, \frac a{g_1}
+\frac {1}{k\sum_{i=1}^s m_i} \right]
$$
with $g_1=GCD(m_1,m_2,\dots, m_s)$.
Again by the induction assumption, for $f_2(x)= \lfloor mx \rfloor -
\lfloor (\sum_{i=1}^s m_i) x \rfloor - \lfloor m_{s+1} x \rfloor$ we have
\begin{equation}
\label{Second}
\int_{0}^u f_2(k x) dx \ge \int_0^u f_2(x) dx
\end{equation}
and the equality holds if and only if 
$$
x\in \bigcup_{b=0}^{\infty}
\left[
\frac b{g_2}-\frac{1}{km}, \frac b{g_2}
+\frac {1}{km} \right]
$$
with $g_2=GCD(\sum_{i=1}^s m_i, m_{s+1})$. 
Since $f(x)=f_1(x)+f_2(x)$, we obtain
$$
\int_{0}^u f(k x) dx \ge \int_0^u f(x) dx.
$$
from (\ref{First}) and (\ref{Second}).
Noting that $g_1,g_2$ are coprime, if either $a/g_1$ or $b/g_2$ is not 
an integer, then $|a/g_1-b/g_2|\ge 1/(g_1g_2)$. The inequality 
(\ref{SmallGCD}) shows that
$$
\left[
\frac a{g_1}-\frac{1}{k\sum_{i=1}^s m_i}, \frac a{g_1}
+\frac {1}{k\sum_{i=1}^s m_i} \right] \cap
\left[
\frac b{g_2}-\frac{1}{km}, \frac b{g_2}
+\frac {1}{km} \right]
\neq \emptyset
$$
if and only if $a/g_1=b/g_2\in \Z$. We have shown the Lemma.
\end{proof}

\begin{lem}
\label{Value}
The function $f(x)$ in Lemma \ref{MainIneq} takes values in $
\{0,1,\dots, s-1\}$.
\end{lem}

\begin{proof}
The case $s=2$ is shown in the proof of Lemma \ref{2Ineq}.
Using the decomposition $f(x)=f_1(x)+f_2(x)$ in the proof of Lemma \ref{MainIneq},
the assertion is easily shown by induction on $s$.
\end{proof}

\section{Proof of Theorem \ref{Main}}
\label{MainProof}

Let $p$ be a prime and $\nu_p(n)$ be
the largest exponent $e$ such that $p^e$ divides $n$.
Using the Legendre formula:
$
\nu_p(n!)=\sum_{e=1}^{\infty} \lfloor n/p^e\rfloor,
$
we see that
$$
\nu_p\left(\prod_{n=1}^t \frac {a_{kn}}{a_n}\right)
= \sum_{e=1}^{\infty} \sum_{n=1}^t \left(
f\left(\frac {kn}{p^e}\right)-f\left(\frac {n}{p^e}\right) \right)
$$
where $f$ is defined in Lemma \ref{MainIneq}. 
Define
$$
H(t):=\sum_{n=1}^t \left(
f\left(\frac {kn}{p^e}\right) - f\left(\frac {n}{p^e}\right) \right).$$
It suffices to prove that
$H(t)\ge 0$ for all $t$ provided $p^e$ is sufficiently large.
First we assume that $p$ and $k$ are coprime. 
Observe that
$$
\frac 1{p^e}
\sum_{n=1}^t \left(f\left(\frac {kn}{p^e}\right)- f\left(\frac {n}{p^e}\right) \right)
$$
is a Riemann sum of the integral 
$$\int_{0}^{t/p^e} \left(f(kx) - f(x)\right) dx.$$ 
Our strategy is to
show that these approximation
is enough fine and Lemma \ref{MainIneq} gives the answer to our problem.  
Since $k$ and $p$ are coprime, 
we see that 
\begin{equation}
\label{Odd}
f(x)+f(1-x)=f(kx)+f(k(1-x))
\end{equation}
 holds for $x\in \frac 1{p^e} \Z \setminus \Z$. 
Indeed $\lfloor x\rfloor+\lfloor -x \rfloor$ is equal to 
$-1$ or $0$ by (\ref{01}),
both sides are equal to $s-1+\Delta(p^e)$ with
$$
\Delta(p^e)= \delta( p^e \mid m ) - \sum_{i=1}^s \delta(p^e \mid m_i)
$$
for $x\in \frac 1{p^e} \Z \setminus \frac 1{p^{e-1}}\Z$ with $e\ge 1$.
Here $\delta(\cdot)$ takes values $1$ or $0$ according to whether 
the inside statement is true or not. 
We claim that
\begin{equation}
\label{Pal}
H(t)=H(p^e-1-t).
\end{equation}
In fact, from (\ref{Odd}) we see
$$
\sum_{n=1}^t \left(
f\left(\frac {kn}{p^e}\right) -f\left(\frac n{p^e}\right) \right)
+
\sum_{n=p^e-t}^{p^e-1} \left(
f\left(\frac {kn}{p^e}\right) -f\left(\frac n{p^e}\right) \right)=0
$$
and thus 
$$
\sum_{n=1}^{p^e-1} \left(
f\left(\frac {kn}{p^e}\right) -f\left(\frac n{p^e}\right) \right)=0.
$$
Therefore we have
$$
\sum_{n=1}^t \left(
f\left(\frac {kn}{p^e}\right) -f\left(\frac n{p^e}\right) \right)
=
\sum_{n=1}^{p^e-1-t} \left(
f\left(\frac {kn}{p^e}\right) -f\left(\frac n{p^e}\right) \right)
$$
which shows the claim. Put $\langle x \rangle=x-\lfloor x\rfloor$. 
If $\langle t/p^e \rangle <1/m$, then $H(t)\ge 0$ is clearly
true because $f(x)=0$ for $0\le x< 1/m$. By (\ref{Pal}), $H(t)\ge 0$
also holds when $\langle t/p^e \rangle >1-1/m$. From Lemma \ref{MainIneq},
there is a positive constant $D$ such that
$$
\int_{0}^u (f(kx)-f(x)) dx \ge D
$$
for $1/m\le u\le (m-1)/m$. 
We have
\begin{eqnarray}
&& 
\nonumber
\int_{0}^{t/p^e} (f(kx)-f(x)) dx
-
\frac 1{p^e}
\sum_{n=1}^t \left(f\left(\frac {kn}{p^e}\right)- f\left(\frac {n}{p^e}\right) \right)\\
\label{Riemann}
&=& \sum_{n=1}^t \left\{
\int_{(n-1)/p^e}^{n/p^e} (f(kx)-f(x)) dx 
-\frac 1{p^e} \left(f\left(\frac {kn}{p^e}\right)- f\left(\frac {n}{p^e}\right)\right)
\right\}.
\end{eqnarray}
Since $f(kx)-f(x)$ is a step function, the last summand is zero 
if the interval $((n-1)/p^e,n/p^e]$ contains no discontinuities, and 
its modulus is bounded from above by $2(s-1)/p^e$ in light of Lemma \ref{Value}.
Letting $E$
be the number of discontinuities of $f(kx)-f(x)$
in $(0,1)$, we have
$$
\left|
\int_{0}^{t/p^e} (f(kx)-f(x)) dx
-
\frac 1{p^e}
\sum_{n=1}^t \left(f\left(\frac {kn}{p^e}\right)- f\left(\frac {n}{p^e}\right) \right)
\right|< \frac{2(s-1)E}{p^e}.
$$
Therefore if $p^e > \frac{2(s-1)E}D$ then
$$
\frac 1{p^e}\sum_{n=1}^t \left(f\left(\frac {kn}{p^e}\right)- f\left(\frac {n}{p^e}\right) \right)
>D- \frac{2(s-1)E}{p^e} >0.
$$
Exceptional discussion is required when $p$ divides $k$. 
Putting $k'=k/p^{\nu_p(k)}$, we have
$$
\sum_{e=1}^{\infty} \left(
f\left(\frac {kn}{p^{e}}\right)-f\left(\frac {n}{p^e}\right) \right) 
=
\sum_{e=1}^{\infty} \left(
f\left(\frac {k'n}{p^{e}}\right)-f\left(\frac {n}{p^e}\right) \right)
$$
since $f(x)=0$ for $x\in \Z$. 
If $k$ is a power of a prime $p$, then $k'=1$ and the right side is identically zero.
If not, we have to replace $k$ by $k'$ and apply the same discussion. 
Then
corresponding $D'$ and $E'$ are computed and we see that
if $K_p:=\frac{2(s-1)E'}{D'}<p^e$ then
$$
\frac 1{p^e}\sum_{n=1}^t \left(f\left(\frac {k'n}{p^e}\right)- f\left(\frac {n}{p^e}\right) \right) >0.
$$
Therefore $H(t)<0$ happens only when
$$
p^e\le M:=
\max \left\{\frac{2(s-1)E}D, \max_{p \mid k \text{ and } k'>1} K_p \right\},
$$ 
which proves Theorem \ref{Main}. See \S \ref{Comp} for the
actual computation of $C(k)$.

\section{Proof of Theorem \ref{Main2}}

We follow the same notation as in the previous section. It suffices to prove
that for any fixed $k$ there are infinitely many pairs $(p,t)$ such that
$$
\sum_{n=1}^t \sum_{e=1}^{\infty} \left(
f\left( \frac {kn}{p^e} \right) -f\left( \frac {n}{p^e} \right) \right)<0.
$$
Take the minimum gap $\kappa$ between two adjacent
discontinuities of $f(kx)$.  We shall find infinitely many such prime $p$'s
which are greater than $\max \{1/\kappa, k^2m\}$
and coprime with $k LCM (m,m_1,m_2\dots,m_s)$. 
Since $p>km$, a rational number with denominator $p$
can not be a discontinuity of $f(kx)$.

From the assumption $g=\mathop{GCD}(m_1,m_2,\dots,m_s)
>1$, we know
$$
\int_{0}^u \left(f(kx)-f(x)\right) dx=0
$$
for $u\in \left[
\frac 1g-\frac{1}{km}, \frac 1g+\frac {1}{km} \right]$ by Lemma \ref{MainIneq}.
Applying (\ref{Riemann}) with $e=1$ and $t=\lfloor p/g\rfloor$, 
since $p>k^2m$ implies $1/km>kt/p^2$ we have  
\begin{eqnarray}
\nonumber
&&
\frac 1p
\sum_{n=1}^t \sum_{e=1}^{\infty} \left(
f\left( \frac {kn}{p^e} \right) -f\left( \frac {n}{p^e} \right) \right)\\
\nonumber
&=&
\frac 1p \sum_{n=1}^t \left(
f\left( \frac {kn}{p} \right) -f\left( \frac {n}{p} \right)\right)\\
&=&
\label{RS}
\sum_{n=1}^t \left\{
\int_{(n-1)/p}^{n/p} \left(f\left(\frac {kn}{p}\right)-f(kx)\right)-
\left(f\left(\frac {n}{p}\right)-f(x)\right) dx 
\right\}.
\end{eqnarray}
For the moment, we tentatively 
think that no discontinuities of $f(kx)$ intersects, i.e.,
$$
\left\{\left. \frac {j}{km}\ \right|\ j=1,\dots, km \right\}
\bigcup \left(\bigcup_{i=1}^s 
 \left\{\left. \frac {j}{km_i}\ \right|\ j=1,\dots, km_i \right\} \right)
$$
are the set of `distinct' $km+\sum_{i=1}^s km_i=2km$ points, 
and
compute the right side. 
Since we are choosing a large $p$, there is at most one
discontinuity $\xi$ in the interval $((n-1)/p,n/p]$. 
If there is no discontinuity in $((n-1)/p,n/p]$, then 
$f\left(\frac {kn}{p}\right)-f(kx)=0.$ 
If such a discontinuity $\xi$ exists, then
we have
\begin{equation}
\label{01m1}
f\left(\frac {kn}{p}\right)-f(kx)=\begin{cases} 0 & x\in [\xi, \frac np] \\
                                                  \pm 1 & x\in (\frac{n-1}p, \xi)\end{cases}
\end{equation}
where $\pm 1$ is $+1$ if $\xi$ is the discontinuity of $\lfloor km x \rfloor$
and $-1$ if $\xi$ is the discontinuity of $\lfloor km_i x \rfloor$
for some $i$. Then we see
\begin{equation}
\label{FracExp}
p\int_{(n-1)/p}^{n/p} \left(f\left(\frac {kn}{p}\right)-f(kx)\right) dx
= \pm \left(p \xi -(n-1)\right) = \pm \langle p \xi \rangle.
\end{equation}

A similar formula holds for $f(n/p)-f(x)$. 
Summing up, from (\ref{RS}) we have shown:
\begin{eqnarray}
\nonumber
&&\sum_{n=1}^t \left(
f\left( \frac {kn}{p} \right) -f\left( \frac {n}{p} \right)\right)\\
\label{Discrete}
&&=
\sum_{j=1}^{km/g} \left\langle \frac {pj}{km} \right\rangle
-\sum_{i=1}^s \sum_{j=1}^{km_i/g} \left\langle \frac {pj}{km_i} \right\rangle
-\left(\sum_{j=1}^{m/g} \left\langle \frac {pj}{m} \right\rangle
-\sum_{i=1}^s \sum_{j=1}^{m_i/g} \left\langle \frac {pj}{m_i} \right\rangle
\right).
\end{eqnarray}
In reality, 
the discontinuities of $f(kx)$ intersect in many places. For e.g., at least
$j/k$ for $j\in \Z$ is a common discontinuity of 
$\lfloor km_i x \rfloor$ and $\lfloor km_j x \rfloor$ for $i\neq j$. 
However
the above formula is correct without any changes. 
This is seen by a similar convention as in the proof
of Lemma \ref{2Ineq}. 
For e.g., if $\xi$ belongs to two discontinuities
of
$\lfloor km_i x \rfloor$ and
$\lfloor km_j x \rfloor$ with $i\neq j$, then 
$$
f\left(\frac {kn}{p}\right)-f(kx) = -2
$$
in $((n-1)/p,\xi)$ instead of (\ref{01m1}), but we computed 
integrand of (\ref{FracExp})
twice in (\ref{Discrete}). 

Put $L=\mathop{LCM}(m,m_1,\dots,m_s)$
and $L'=\mathop{LCM}(m_1,\dots,m_s)$.
An importance of the formula (\ref{Discrete}) is that 
the value is 
determined by $p \pmod{kL}$. 
As $p$ is coprime to $kL$, 
by Dirichlet's theorem on primes in arithmetic progression, it suffices to show
that there exists a single $p$ 
such that the value of (\ref{Discrete}) 
is negative to prove our theorem.
Noting 
\begin{equation}
\label{Zero2}
\sum_{j=1}^{km/g} \frac {pj}{km} 
-\sum_{i=1}^s \sum_{j=1}^{km_i/g} \frac {pj}{km_i} 
-\left(\sum_{j=1}^{m/g} \frac {pj}{m} 
-\sum_{i=1}^s \sum_{j=1}^{m_i/g} \frac {pj}{m_i} \right)=0,
\end{equation}
it is equivalent to show that there is such a $p$ that
\begin{equation}
\label{IntegerPart}
\sum_{j=1}^{km/g} \left\lfloor \frac {pj}{km} \right\rfloor
-\sum_{i=1}^s \sum_{j=1}^{km_i/g} \left\lfloor \frac {pj}{km_i} \right\rfloor
-\left(\sum_{j=1}^{m/g} \left\lfloor \frac {pj}{m} \right\rfloor
-\sum_{i=1}^s \sum_{j=1}^{m_i/g} \left\lfloor \frac {pj}{m_i} \right\rfloor
\right)
\end{equation}
is positive\footnote{From this expression, we see that (\ref{Discrete}) is
integer valued.}.
Moreover using 
$\lfloor x\rfloor+\lfloor -x \rfloor$ is equal to a constant 
$-1$ for non integer $x$ by (\ref{01}), 
the value of $(\ref{IntegerPart})$  
changes its sign
by the involution $p \leftrightarrow kL-p$ for $p$ which is coprime to $kL$,
and the same holds for (\ref{Discrete}).
Therefore our task is to show that either (\ref{Discrete}) or (\ref{IntegerPart})
is not zero for some $p$ which is coprime to $kL$. 
We show this by dividing into three cases. 

{\it Case $m>L'$}. One can find a $p$ that
$$
p\equiv 1 \pmod{kL'}, \qquad 
p\not \equiv 1 \pmod {km}
$$
and is coprime with $kL$.
From (\ref{Zero2}), the right side of (\ref{Discrete}) becomes
\begin{eqnarray}
\nonumber
&&
\sum_{j=1}^{km/g} \left\langle \frac {pj}{km} \right\rangle
-\sum_{i=1}^s \sum_{j=1}^{km_i/g} \frac {j}{km_i} 
-\left(\sum_{j=1}^{m/g} \left\langle \frac {pj}{m} \right\rangle
-\sum_{i=1}^s \sum_{j=1}^{m_i/g} \frac {j}{m_i}
\right)\\
\nonumber
&=&
\sum_{j=1}^{km/g} 
\left(\left\langle \frac {pj}{km} \right\rangle
-\frac {j}{km} \right)
-\sum_{j=1}^{m/g} \left(\left\langle \frac {pj}{m} \right\rangle
-\frac {j}{m}
\right)\\
\label{Sum1}
&=&
\sum_{\substack{j=1\\ j\not \equiv 0 \pmod{k}}}^{km/g} 
\left(\left\langle \frac {pj}{km} \right\rangle
-\frac {j}{km} \right).
\end{eqnarray}
Here we have
$$
\sum_{\substack{j=1\\ j\not \equiv 0 \pmod{k}}}^{km/g} 
\left\langle \frac {pj}{km} \right\rangle
=
\frac {1}{km} \sum_{\substack{j=1\\ j\not \equiv 0 \pmod{k}}}^{km/g} 
\mathcal{R}(pj,km),
$$
where $\mathcal{R}(a,b)$ is the minimum non negative integer congruent to $a$ mod $b$.
Because the function $\langle x/km\rangle$ is increasing 
for $0\le x< km$, 
the minimum is attained by
$$
\frac {1}{km} \sum_{\substack{j=1\\ j\not \equiv 0 \pmod{k}}}^{km/g} 
\mathcal{R}(j,km)
=
\sum_{\substack{j=1\\ j\not \equiv 0 \pmod{k}}}^{km/g} 
\frac {j}{km}
$$
which shows that (\ref{Sum1}) is non negative. However, by $p\not \equiv 1
\pmod{km}$, clearly $\mathcal{R}(pj,km)$ 
takes values outside $\{\mathcal{R}(j,km)\ 
|\ 1\le j\le km/g, \quad j\not \equiv 0 \pmod{k}\}$
and thus (\ref{Sum1}) must be positive.

{\it Case $m<L'$}. We choose a $p$ that
$$
p\not \equiv 1 \pmod{kL'}, \qquad 
p\equiv 1 \pmod {km}.
$$
From (\ref{Zero2}), the right side of (\ref{Discrete}) becomes
\begin{eqnarray*}
&&
\sum_{j=1}^{km/g} \frac {j}{km} 
-\sum_{i=1}^s \sum_{j=1}^{km_i/g} \left\langle \frac {pj}{km_i} \right\rangle
-\left(\sum_{j=1}^{m/g} \frac {j}{m} 
-\sum_{i=1}^s \sum_{j=1}^{m_i/g} \left\langle \frac {pj}{m_i} \right\rangle
\right)\\
&=&-\sum_{i=1}^s \sum_{j=1}^{km_i/g} \left(\left\langle \frac {pj}{km_i} 
\right\rangle- \frac {j}{km_i}\right)
-\sum_{i=1}^s \sum_{j=1}^{m_i/g} \left(\left\langle \frac {pj}{m_i} \right\rangle
-\frac {j}{m_i}\right)\\
&=&-\sum_{i=1}^s \sum_{\substack{j=1\\j\not \equiv 0 \pmod{k}}}^{km_i/g} 
\left(\left\langle \frac {pj}{km_i} 
\right\rangle- \frac {j}{km_i}\right).
\end{eqnarray*}
The inner sums are 
non negative and at least one of them is positive by the same discussion
as in the former case.

{\it Case\footnote{There are such pairs, for e.g., 
$1+2+3=\mathop{LCM}(1,2,3)$, $2+3+3+4=\mathop{LCM}(2,3,3,4)$.} $m=L'$}.
We rewrite (\ref{IntegerPart}) into
$$
\mathcal{F}(p):=
\sum_{\substack{j=1\\ j \not \equiv 0 \pmod{k}}}^{km/g} \left\lfloor \frac {pj}{km} 
\right\rfloor
-\sum_{i=1}^s \sum_{\substack{j=1\\ j \not \equiv 0 \pmod{k}}}^{km_i/g} \left\lfloor \frac {pj}{km_i} \right\rfloor.
$$
It suffices to show that there are two integers $p_1$ and $p_2$ which are
coprime with
$km=k \mathop{LCM}(m,m_1,\dots,m_s)$ and $\mathcal{F}(p_1)\neq \mathcal{F}(p_2)$. 

First we study the case that there is a prime $q$ with
$q\mid \frac mg$ and $q^2\mid km$. Then we can take
 $p_1=km/q-1$ and $p_2=km/q+1$ which are
coprime with $km$. Since
$$
\left\lfloor \frac jq + \frac j{km} \right\rfloor -\left\lfloor \frac jq - \frac j{km} \right\rfloor 
=\begin{cases} 0 & j\not \equiv 0 \pmod{q}\\
               1 & j\equiv 0 \pmod{q}
\end{cases},
$$
we obtain, 
$$
\mathcal{F}(p_2)-\mathcal{F}(p_1)
=
\sum_{\substack{1\le j \le km/g\\ j \not \equiv 0 \pmod{k}\\ j \equiv 0 \pmod{q}}} 
1-
\sum_{i=1}^s \sum_{\substack{1\le j\le km_i/g\\ j \not \equiv 0 \pmod{k}\\ \frac{m}{m_i}
j \equiv 0 \pmod{q}}} 1.
$$
Since $m=m_1+m_2+\dots+m_s$ and $m/m_i$ are integers, 
the right side is expected to be negative. 
However 
a careful computation is required, because not all $km_i/g$ are divisible by $q$.
Since $q$ divides either $m_i/g$ or $m/m_i$, we have
$$
\sum_{\substack{1\le j\le km_i/g\\ j \not \equiv 0 \pmod{k}\\ \frac{m}{m_i}
j \equiv 0 \pmod{q}}} 1
=\begin{cases} \frac{km_i}g - \frac{m_i}g & m/m_i\equiv 0 \pmod{q} \\
               \frac{km_i}{gq} - \frac{km_i}{g\mathop{LCM}(k,q)} &  
               m/m_i\not\equiv 0 \pmod{q}
\end{cases}.
$$
Therefore
\begin{eqnarray*}
&&\mathcal{F}(p_2)-\mathcal{F}(p_1)=\frac{km}{gq}-\frac{km}{g\mathop{LCM}(k,q)}\\
&&
-\sum_{\substack{1\le i\le s\\ \frac{m}{m_i}\not \equiv 0 \pmod{q}}}
\left(\frac{km_i}{gq} - \frac{km_i}{g\mathop{LCM}(k,q)}\right)
-\sum_{\substack{1\le i\le s\\ \frac{m}{m_i}\equiv 0 \pmod{q}}} 
\left(\frac{km_i}g - \frac{m_i}g\right)\\
&=&\sum_{\substack{1\le i\le s\\ \frac{m}{m_i}\equiv 0 \pmod{q}}}
\left(\left(\frac{km_i}{gq} - \frac{km_i}{g\mathop{LCM}(k,q)}\right)
-\left(\frac{km_i}g - \frac{m_i}g\right)\right)
\\
&=&-\sum_{\substack{1\le i\le s\\ \frac{m}{m_i}\equiv 0 \pmod{q}}}
\frac{m_i}g\frac {(k-1)((\mathop{LCM}(k,q)+1)(q-1)+1)-\mathop{LCM}(k,q)+q
}{q \mathop{LCM}(k,q)}\\
&\le& -\sum_{\substack{1\le i\le s\\ \frac{m}{m_i}\equiv 0 \pmod{q}}}
\frac{(q+2)m_i}{gq \mathop{LCM}(k,q)}<0.
\end{eqnarray*}
Here we used $k\ge 2$ and $q\ge 2$. Note that the last sum is non empty, 
because $q\mid (m/g)$ and $\mathop{GCD}((m_1/g),\dots,(m_s/g))=1$ 
implies that $q$ divides $m/m_i$ for at least one $i$. 


Second, consider the case that there is a prime divisor $q>3$ of $m/g$. 
We may assume that either
$$
(p_1,p_2)=\left(\frac {km}q-1, \frac{km}q+1\right) \text{ or } 
\left(\frac{2km}q-1,\frac{2km}q+1\right)
$$
satisfies $
\mathop{GCD}(p_i, km)=1$ for $i=1$ and $2$.
In fact, if for e.g. $km/q-1\equiv 0 \pmod{q}$
and $2km/q+1\equiv 0 \pmod{q}$ hold, 
then $q^2$ divides $km$, which is reduced to the first case.
Once we have such a pair $(p_1,p_2)$, we can show
$$
\mathcal{F}(p_2)-\mathcal{F}(p_1)<0
$$
in the same manner.
So we finally consider the case 
that all the prime divisors of $m/g$
is $2$ and $3$ and $m/g$ is square free, which covers the remaining cases. 
There exists only one such case 
with
$$
\frac{m_1}g+\dots+\frac{m_s}g=\mathop{LCM} \left(\frac{m_1}g,\dots, \frac{m_s}g\right),
$$
that is, $s=3$ and $(m_1/g,m_2/g,m_3/g)=(1,2,3)$. So our last task is to 
consider the case:
$
(m_1,m_2,m_3)=(g,2g,3g).
$
Since $m=6g$ is even we can choose:
$$
(p_1,p_2)=\begin{cases}(\frac {km}2-1, \frac {km}2+1) & km\equiv 0 \pmod{4}\\
                       (\frac {km}2-2, \frac{km}2+2) & km\equiv 2 \pmod{4}
                       \end{cases}.
$$ 
Then we see $\mathcal{F}(p_2)-\mathcal{F}(p_1)<0$ in the same manner.

\section{Proof of Theorem \ref{Mean}}

This proof is inspired by Theorem 5.2 and Lemma 5.2 in \cite{
Akiyama-Petho:12_1}.
We use the same terminology as in section \ref{MainProof}. Under the assumption,
$f(x)$ has no discontinuity of negative jump in $(0,1)$ by cancellation. 
Thus $f(x)$ is non decreasing. We show that
$$
\sum_{n=1}^{t} \left(
f\left(\frac {k'n}{p^{e}}\right)-f\left(\frac {n}{p^e}\right) \right)\ge 0
$$
where $k'=k/p^{\nu_p(k)}$. By periodicity of $f$, it suffices to show the
case that $t<p^e$. Since $k'$ and $p$ are coprime, $k'n \bmod{p^e}\ (n=1,2,\dots,
p^e-1)$ are 
distinct. From $f(1/p^e)\le f(2/p^e)\le \dots \le f((p^e-1)/p^e)$, we see
$$
\sum_{n=1}^{t} f\left(\frac {n}{p^e}\right) 
$$
is the minimum of the sum of $t$ elements in $\{f(i/p^e)\ |\ i=1,2,\dots,p^e-1\}$,
which finishes the proof.

\section{Computation of $C(k)$}
\label{Comp}
In this section, we explain the computation of the constant $C(k)$ 
by an algorithm based on 
the proof of Theorem \ref{Main}. For a given $k$, first 
we compute the minimum $D_q$ of $I(u)=\int_{0}^u f(k'x)-f(x) dx$ for 
$u\in [1/m,1-1/m]$ 
for $k'=k'(q)=k/q^{\nu_q(k)}$ for all prime divisor $q$ of $k$.
Denote by $D_1$ the minimum of $I(u)$ for $k'=k$.
Since the minimum of $I(u)$ is attained at the discontinuity of the step
function: $I'(u)=f(k'u)-f(u)$, using (\ref{Pal}), 
it is explicitly computed as
$$
\min I(u)= \min \left\{ \min_{1\le j\le \frac m2} 
I\left(\frac {j}m\right), \min_{1\le i\le s} \min_{\frac {k'm_i}m\le 
j\le \frac {k'm_i}2} I\left(\frac j{k'm_i}\right) \right\}.
$$
By Theorem \ref{Main}, we know $D_q>0$. 
Number of discontinuities
is bounded from above\footnote{It is better to take 
the exact value to make faster the computation, as we do below in examples.} by $2k'm$. 
Then we compute for prime $p$'s with
$$
p \le M:=\max \left\{ \frac{2km(s-1)}{D_1}, \max_{q\mid k \text{ and } k'>1} 
\frac{2k'(q)m(s-1)}{D_q} \right\},
$$
the values 
$$
\mu(p)= \min \left\{0, 
\min_{t=1}^{p^h-1} 
\sum_{n=1}^t 
\sum_{e=1}^{h+\ell-1}
\left(f\left(\frac {k'n}{p^e}\right) -f\left(\frac {n}{p^e}\right)\right) \right\}
$$
with $h=\lfloor \log (M)/\log (p)\rfloor$
and $\ell=\lfloor \log (k'm)/\log (p)\rfloor$.
By the proof of Theorem \ref{Main}, if $p^e>M$ then 
$$
\sum_{n=1}^t 
\left(f\left( \frac {k'n}{p^e}\right) -f\left(\frac {n}{p^e}\right)\right) \ge 0
$$
for any positive integer $t$. Further by $\frac {k'n}{p^{h+\ell}}< \frac 1m$, 
$$ 
\sum_{n=1}^t
\left(f\left( \frac {k'n}{p^e}\right) -f\left(\frac {n}{p^e}\right)\right) = 0
$$
for $e\ge h+\ell$ and $t<p^h$. We have
$$
\mu(p)= \min \left\{0, 
\min_{t=1}^{\infty} 
\sum_{n=1}^t
\sum_{e=1}^{\infty}
\left(f\left(\frac {k'n}{p^e}\right) -f\left(\frac {n}{p^e}\right)\right) \right\}
$$
and
$p^{\mu(p)}$ is attained as the denominator for some $t$. 
We obtained 
the constant $C(k)=\prod_{p\le M} p^{|\mu(p)|}$. \qed
\bigskip

We briefly demonstrate this algorithm by
showing Corollary \ref{Co1} and make precise the comments afterwords. 
For $s=2$, $m_1=2$, $m_2=3$ and $k=2$, the graph of
the function $I(u)$ for $k'=k$ is depicted in Figure \ref{23}.
\begin{figure}[h]
\includegraphics{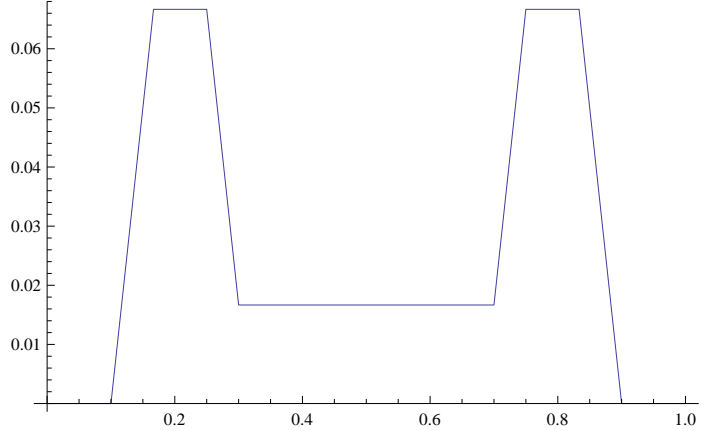}
\caption{$m_1=2,m_2=3,k=2$ \label{23}}
\end{figure}
We may take $D_1=1/60$ and $E_1=8$. Checking all primes 
which do not exceed $2\cdot 8 \cdot 60=960$, we found the only non zero
output $\mu(11)=-1$. We can confirm that
$$
\sum_{n=1}^t \sum_{e=1}^{\infty} \left(f(2n/11^e)-f(n/11^e)\right) =
\sum_{n=1}^3 \left(f(2n/11)-f(n/11)\right)=-1
$$
when $t=3+11^h$ and $h\ge 2$, since $f(x)=0$ for $0\le x\le 1/5$.
Therefore the denominator $11$ in the first formula in 
Corollary \ref{Co1} appears infinitely often.
For $s=2$, $m_1=1$, $m_2=2$ and $k=3$, the graph of $I(u)$ 
is depicted in Figure \ref{12}.
\begin{figure}[h]
\includegraphics{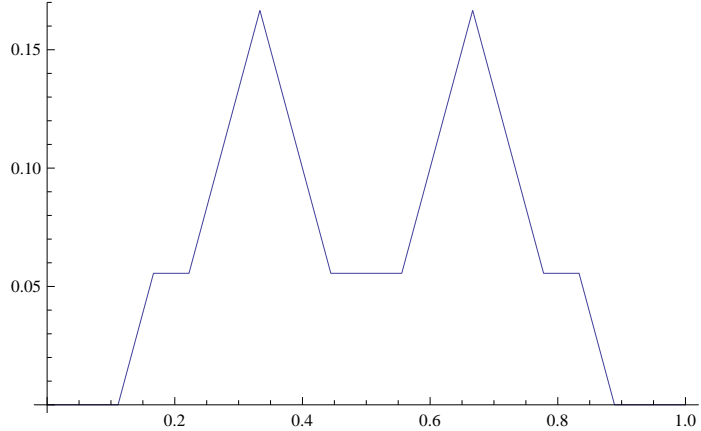}
\caption{$m_1=1,m_2=2,k=3$ \label{12}}
\end{figure}
We have 
$D_1=1/18$, $E_1=10$ and the only non zero $\mu$-value is $\mu(5)=-1$.
In this case, one can also confirm that
$$
\sum_{n=1}^{2+5^h} \sum_{e=1}^{h} \left(f(3n/5^e)-f(n/5^e)\right) =-1.
$$
However we have
$$
\sum_{n=1}^t \sum_{e=1}^{h+1} \left(f(3n/5^e)-f(n/5^e)\right) \ge 0
$$
for $2<t<5^h-1$ and any positive integer $h$. 
In other words, the function 
$\sum_{n=1}^t \sum_{e=1}^{h} \left(f(3n/5^e)-f(n/5^e)\right)$ has a period $5^h$ and 
attains $-1$ infinitely often as above, but such negative values are erased 
by the next period of length $5^{h+1}$, except when $t=2$. 
The denominator $5$ in the second formula of Corollary \ref{Co1}
appears only when $t=2$. 

For $s=3$, $m_1=1$, $m_2=2$, $m_3=4$ and $k=4$, the graph of
is depicted in Figure \ref{124}.
\begin{figure}[h]
\includegraphics{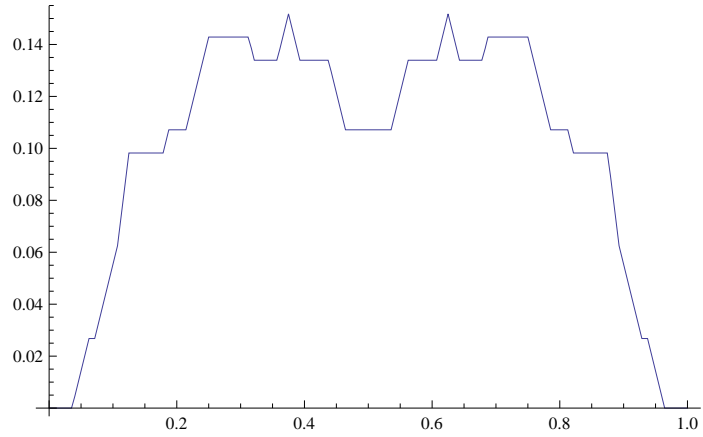}
\caption{$m_1=1,m_2=2,m_3=4,k=4$ \label{124}}
\end{figure}
We have 
$D_1=11/112$, $E_1=32$ and $\mu(p)=0$ for all prime $p$.

In this manner, a prime divisor $p$ of $C(k)$ actually
appears in the denominator  only 
when the $p$-adic expansion of $t$ has 
a special form, and not easy to describe the set of such $t$'s. 

\section{Questions}
\label{Open}

We wish to list several open problems. 
A sequence $(a_n)$ may be called {\it almost mean divisible} 
if there is a 
positive integer $C$ such that 
$$(\prod_{n=1}^t a_{kn})/(\prod_{i=1}^t a_n)\in \frac 1C \Z$$ 
for any positive integers $k$ and $t$. In other words, 
$(a_n)$ is almost mean divisible, 
if it is almost mean $k$-divisible
with a uniform constant $C$ independent of the choice of $k$. 

\begin{itemize}
\item Is there an 
almost mean divisible multinomial sequence, which is not mean divisible ?
\item Is there a mean divisible multinomial sequence which does not
satisfy the condition of Theorem \ref{Mean} ?
\item Is there any other (almost)
 mean divisible sequence of number theoretical interests ?
\end{itemize}

I expect the answer for the first question is negative, because 
the bound $M$ increases as
$k$ becomes large, in the proof of Theorem \ref{Main}. For the second, 
there may exist such multinomial sequences for $s>2$. 
For $s=3$ and $(m_1,m_2,m_3)
=(m,1,1)$, I checked by the algorithm in \S \ref{Comp} to
obtain a
\begin{cor}
If $3\le m\le 10$ and $k\le 10$, then
$$
\prod_{n=1}^t \frac 
{\binom{(m+2)kn}{kmn,kn,kn}}
{\binom{(m+2)n}{mn,n,n}}\in \Z
$$
holds\footnote{By Theorem \ref{Mean}, this is valid 
for all positive integer $k$ when $m=1$ and $2$.} for all positive integer $t$.
\end{cor}
We do not know if this is true for all $k$ for some $m\ge 3$.
This sequence is
factored into two:
$$
\binom{(m+2)n}{mn,n,n}=\binom{(m+2)n}{2n} \binom{2n}{n}.
$$
The former sequence $(\binom{(m+2)n}{2n})$ is almost mean $k$-divisible
for all $k$ and an odd $m$ 
by Theorem \ref{Main} and $(\binom{2n}{n})$ is mean divisible 
by Theorem \ref{Mean}. So the denominators generated by the first sequence 
might be canceled by the numerators from the later one. It is an interesting problem
to characterize all mean divisible multinomial sequences.

As for the third question, we can construct a different type of
non divisible almost mean $k$-divisible sequences.
Fix an integer $\ell>1$ and let 
$\alpha$ and $\beta$ be conjugate quadratic integers so that $\alpha/\beta$ is
not a root of unity. 
Define the $\ell$-th homogeneous cyclotomic polynomial:
$$
\Phi_{\ell}(x,y)=\prod_{\substack{0< m<\ell\\
\mathop{GCD}(m,\ell)=1}} \left(x -\zeta^m y\right)
$$
where $\zeta$ is the primitive $\ell$-th root of unity.
Put
$$
c_n=\Phi_{\ell} \left(\alpha^n,\beta^n\right).
$$ 
Then $(c_n)$ is a non zero integer sequence and 
for any integer $k$ coprime to $\ell$, we have $c_n \mid  c_{kn}$. 
This implies that $(c_n)$ is almost mean $k$-divisible for $\mathop{GCD}(k,\ell)=1$. 
For example, taking 
$\ell=2$, 
$$
L_n=\left(\frac{1+\sqrt{5}}2\right)^n+\left(\frac{1-\sqrt{5}}2\right)^n
$$
gives the Lucas sequence, 
which is non divisible but almost mean $k$-divisible for all odd integer $k$.
However, it may not be a significant construction because they
already have divisibility $c_n \mid  c_{kn}$ not for all but for some $k$.
\bigskip

{\bf Acknowledgments.}
The author would like to express his gratitude to Attila Peth\H o 
for stimulating discussion and relevant references. He is also deeply
grateful to Andrew Granville
and the anonymous referee who supplied him good references and advices on 
readability and further directions.


\end{document}